\documentclass[a4paper, english, 10pt, onecolumn, oneside]{article}	
\usepackage[verbose, a4paper, hcentering]{geometry}	

\usepackage[dvipsnames]{xcolor}
\definecolor{greentuc}{rgb}{0.0, 0.38, 0.30}

\usepackage[utf8]{inputenc} 
\usepackage[T1]{fontenc}    
\usepackage[pdftitle={Soft quantization}, pdfauthor={Rajmadan Lakshmanan}, citecolor=greentuc, backref=page, urlcolor= greentuc, linkcolor= greentuc, unicode=true,  bookmarks=true, bookmarksnumbered=false, bookmarksopen=false, breaklinks=false, pdfborder={0 0 0}, pdfborderstyle={}, colorlinks=true]{hyperref}
\usepackage[numbers]{natbib}	
\usepackage{siunitx} 
\usepackage{booktabs}       
\usepackage{amsthm, mathtools, amssymb, nicefrac}
\usepackage{amsmath}
\numberwithin{equation}{section}
\usepackage{unicodeMath}
\mathtoolsset{showonlyrefs}

\usepackage{microtype}      
\usepackage{newtxtext, newtxmath}
\usepackage{dsfont}   
\usepackage{tikz}
\usetikzlibrary{calc,trees,positioning,arrows,fit,shapes,calc}
\usepackage{pgfplots}
\usepackage{subfig}
\pgfplotsset{compat=newest}

\usepackage[inline, shortlabels]{enumitem}           
\setlist[enumerate,1]{label=(\roman*)}  

\usepackage{xcolor}         
\usepackage{siunitx}
\usepackage{todonotes}
\usepackage[export]{adjustbox}
\usepackage{mathtools}
\usepackage{subfig}
\usepackage{graphicx}

\usepackage{pgfplots}
\usepackage{siunitx}
\usepackage{tikz} 
\usetikzlibrary{trees,matrix,arrows.meta}
\usepackage[ruled, vlined]{algorithm2e}			
\usepackage{orcidlink}
\usepackage{diagbox}
\usepackage{doi}
\usepackage{pgfplots}
\usepackage{xurl}

\DeclareMathOperator{\E}{{\mathds E}}
\DeclareMathOperator{\diag}{diag}		

\DeclareMathOperator*{\essinf}{ess\,inf}
\DeclareMathOperator*{\argmin}{arg\,min}
\DeclareMathOperator{\supp}{{supp}}
\DeclareMathOperator{\sign}{{sign}}
\newcommand{\smin}[1]{\mathop{\min\nolimits_{#1}}}

\newcommand{\one}{{\mathds 1}} 		

\newcounter{relctr} 
\everydisplay\expandafter{\the\everydisplay\setcounter{relctr}{0}} 

\AtBeginDocument{}

\theoremstyle{plain}
	\newtheorem{theorem}{Theorem}[section]		
		
	\newtheorem{lemma}[theorem]{Lemma}
	\newtheorem{proposition}[theorem]{Proposition}
\theoremstyle{definition}
	\newtheorem{definition}[theorem]{Definition}
\theoremstyle{remark}
	\newtheorem{remark}[theorem]{Remark}
	\newtheorem{example}[theorem]{Example}

\title{Soft Quantization using Entropic Regularization}
\author{
	Rajmadan Lakshmanan%
		\thanks{Technische Universität Chemnitz, Faculty of mathematics, Chemnitz, Germany}\,
		 \footnote{\orcidlink{0009-0006-3273-9063} \href{https://orcid.org/0009-0006-3273-9063}{https://orcid.org/0009-0006-3273-9063}. Contact: \protect\href{rajmadan.lakshmanan@math.tu-chemnitz.de}{rajmadan.lakshmanan@math.tu-chemnitz.de}}\and
	Alois Pichler\footnotemark[1]\ \ %
		\thanks{\orcidlink{0000-0001-8876-2429} \href{https://orcid.org/0000-0001-8876-2429}{https://orcid.org/0000-0001-8876-2429}.
		Contact: \protect\href{mailto:alois.pichler@math.tu-chemnitz.de}{alois.pichler@math.tu-chemnitz.de} \\
		DFG, German Research Foundation – Project-ID 416228727 – SFB~1410}
}

\begin{document}
\maketitle
	\begin{abstract}
		The quantization problem aims to find the best possible approximation of probability measures on $ℝᵈ$ using finite, discrete measures.
		The Wasserstein distance is a typical choice to measure the quality of the approximation.

		This contribution investigates the properties and robustness of the entropy-regularized quantization problem, which relaxes the standard quantization problem.
		The proposed approximation technique naturally adopts the softmin function, which is well known for its robustness in terms of theoretical and practicability standpoints.
		Moreover, we use the entropy-regularized Wasserstein distance to evaluate the quality of the soft quantization problem's approximation, and we implement a stochastic gradient approach to achieve the optimal solutions.
		The control parameter in our proposed method allows for the adjustment of the optimization problem's difficulty level, providing significant advantages when dealing with exceptionally challenging problems of interest.
		As well, this contribution empirically illustrates the performance of the method in various expositions.\medskip
		

		\noindent \textbf{Classification:}
		94A17, 81S20, 40A25 
	\end{abstract}

\section{Introduction}

Over the past few decades, extensive research has been conducted on optimal quantization techniques in order to tackle numerical problems that are related to various fields such as data science, applied disciplines, and economic models. 
These problems are typically centered around \emph{uncertainties} or \emph{probabilities} which demand robust and efficient solutions (cf.␠\citet{Graf1989}, \citet{luschgy2015greedy}, \citet{el2022new}).
In general, these problems are difficult to handle, as the random components in the problem allow uncountable many outcomes.
As a consequence to address this difficulty, the probability measures are replaced by simpler or finite measures, which facilitates numerical computations.
However, the probability measures should be ‘close’, so that the result of the computations with approximate (discrete) measures will resemble the original problem. 
In a nutshell, the goal is to find the best approximation of a diffuse
measure using a discrete measure, and it is called \emph{optimal quantization} problem.
For a comprehensive discussion of the optimal quantization problem from a mathematical standpoint, we refer to \citet{GrafLuschgy}.

On the other hand, \emph{entropy} is an inevitable concept to deal with uncertainties and probabilities. 
In mathematics, entropy is often used as a measure of information and uncertainty. It provides a quantitative measure of the randomness or disorder in a system or a random variable. 
Its applications span across information theory, statistical analysis, probability theory, and the study of complex dynamical systems (cf. \citet{Breuer2013,Breuer2013a}, \citet{PichlerSchlotterEntropy}).

In order to assess the closeness of the probability measures, distances are often considered, and one of the notable instances is the Wasserstein distance.
Ostensibly, the Wasserstein distance measures the minimum, average amount of transporting cost required to transfer one probability distribution into another. 
Unlike other formulations of distances and/ or divergence, which simply compares the probabilities of the distribution functions (e.g., the total variation distance and the Kullback–Leibler divergence), the Wasserstein distance incorporates the support of the underlying distributions.
This increases the understanding of the relationships between different probability measures in a geometrically trustworthy manner.

In our research work, we focus on entropy adjusted quantization methods.
More precisely, we consider an entropy regularized version of the Wasserstein problem to quantify the quality of the approximation, and we adapt the stochastic gradient approach to obtain the optimal quantizers.

Some key features of our methodology include the following:
\begin{enumerate}[]
	◦	This regularization approach stabilizes and simplifies the standard quantization problem by introducing penalty terms or constraints that discourage overly complex or overfit models, promoting better generalizations and robustness in the solutions.
	◦	The influence of entropy is controlled using a parameter $λ$, which also facilitates us to reach the genuine optimal quantizers.
	◦	Generally, parameter tuning comes with certain limitations. However, our method builds upon the framework of the well-established softmin function, which allows us to exercise parameter control without encountering any restrictions.
	◦	For larger regularization parameter $λ$, the optimal measure accumulates all its mass at the center of the measure.
\end{enumerate}

\paragraph{Related works and contributions.} 
As mentioned above, optimal quantization is a well-researched topic in the field of information theory and signal processing. There are several methods that have been developed for optimal quantization problem. Here are some remarkable methods of optimal quantization:
\begin{enumerate}[label=--]
	◦	Lloyd-Max Algorithm: the Lloyd-Max algorithm, also known as the Lloyd’s algorithm or the $k$‑means algorithm, is a popular iterative algorithm for computing optimal vector quantizers. It iteratively adjusts the centroids of the quantization levels to minimize the quantization error (cf. \citet{scheunders1996genetic}).
	◦	Tree-Structured Vector Quantization (TSVQ): TSVQ is a hierarchical quantization method that uses a tree structure to partition the input space into regions. It recursively applies vector quantization at each level of the tree until the desired number of quantization levels is achieved (cf. \citet{wei2000fast}).
	◦	Expectation-maximization (EM) algorithm: the EM algorithm is a general-purpose optimization algorithm that can be used for optimal quantization. It is an iterative algorithm that estimates the parameters of a statistical model to maximize the likelihood of the observed data (cf. \citet{heskes2001self}).
	◦	Stochastic Optimization Methods: Stochastic optimization methods, such as simulated annealing, genetic algorithms, and particle swarm optimization, can be used to find optimal quantization strategies by exploring the search space and iteratively improving the quantization performance (cf. \citet{stochasticPages}).
	◦	Greedy vector quantization (GVQ): the greedy algorithm tries to solve this problem iteratively, by adding one code word at every step until the desired number of code words is reached, and each time selecting the code word that minimizes the error. GVQ is known to provide suboptimal quantization compared to other non-greedy methods like the Lloyd-Max and Linde-Buzo-Gray algorithms. However, it has been shown to perform well when the data has a strong correlation structure. Notably, it utilizes Wasserstein distance to measure the error of approximation (cf.␠\citet{luschgy2015greedy}).
\end{enumerate}

These methods provide efficient and practical solutions for finding optimal quantization schemes with different trade-offs between complexity and performance. The choice of method depends on the problem of interset and the requirements of the application. However, most of these methods depend on strict constraints which makes the solutions overly complex or overfit models. 
Our method mitigates this issue by promoting better generalizations and robustness in the solutions.

In the optimal transport community, entropy regularized version of optimal transport problem (also known as entropy regularized Wasserstein problem) is initial proposed by \citet{cuturi2013sinkhorn}.
This entropy version of Wasserstein problem promotes fast computations using Sinkhorn’s algorithm. 
As an avenue for constructive research, this study has presented a multitude of results aimed at gaining a comprehensive understanding of the subtleties involved in enhancing the computational performance of entropy optimal transport (cf.\ \citet{ramdas2017wasserstein}, \citet{neumayer2021optimal}, \cite{NEURIPS2019_f55cadb9}, \citet{lakshmanan2023nonequispaced}).
These findings serve as a valuable foundation for further exploration in the field of optimal transport, providing insights into both the intricacies of the topic and potential avenues for improvement.

In contrast, we present a new, innovative approach that concentrates on the optimal quantization problem based on entropy, and its robust properties, which is a distinct contribution from standard entropy regularized optimal transport problems.
\medskip

One of the principal consequences of our research substantiates the behavior of convergence of quantizers at the center of the measure.
The relationship between the center of measure and entropy regularized quantization problem has not been exposed yet.
The following plain solution is obtained by intensifying the entropy term in the regularization of the quantization problem.
\begin{theorem}\label{thm:198}
	There exist a real valued $λ₀> 0$ such that the best approximation of entropy regularized optimal quantization problem is given by the Dirac-measure
	\[	P= δₐ \]
	for every $λ> λ₀$, where $a$ is the center of the measure $P$ with respect to the distance $d$.
\end{theorem}
The enthralling interpretation of our master problem facilitates us to understand the transition from a complex hard optimization solution to the simple solution in Theorem~\ref{thm:198}. 
Moreover, along with the theoretical discussion, we provide an algorithm and numerical exemplification, which empirically demonstrate the robustness of the method. The forthcoming sections elucidate the robustness and asymptotic properties of the methods in detail. 



\paragraph{Outline of the paper.}
Section \ref{sec:Preliminaries} establishes the essential notations, definitions, and properties. Moreover, we comprehensively expound upon the significance of the smooth minimum, a pivotal component in our research. 
In Section~\ref{sec:Quantization}, we introduce the entropy-regularized optimal quantization problem and delve into its inherent properties.
Section \ref{sec:Tessellation} presents the discussion of soft tessellation, optimal weights and theoretically properties of parameter tuning.
Furthermore, we systematically illustrate the computational process along with a pseudo algorithm.
Section \ref{sec:Numerical} provides numerical examples, and empirically substantiates the theoretical proofs.
Finally, Section \ref{sec:Summary} summarize our study.

\newpage
%

\section{Preliminaries\label{sec:Preliminaries}}


In what follows, $(𝓧,d)$ is a Polish space.
The $σ$‑algebra generated by the Borel sets induced by the distance $d$ is $𝓕$, the set of all probability measures on $𝓧$ is $𝒫(𝓧)$.

\subsection{Distances and divergences of measures}\label{sec:Distances}
The standard quantization problem employs the Wasserstein distance to measure the quality of the approximation, which was initially studied by Monge and Kantorovich (cf.␠\citet{Monge1781}, \citet{Kantorovich1942}).
One of the remarkable properties of this distance is, it metrizes the weak* topology of measures.
\begin{definition}[Wasserstein distance]
	Let $P$ and $Ꝓ$ be probability measures on $(𝓧,d)$.
	The Wasserstein distance of order $r≥ 1$ of $P$ and $Ꝓ∈ 𝒫 (𝓧)$ is
	\begin{equation}
		dᵣ(P,Ꝓ)≔ \inf❪∬_{𝓧⨯𝓧} d(ξ,𝜉)ʳ π(ⅾξ,ⅾ𝜉)❫^{\nicefrac1r},
	\end{equation}
	where the infimum is among all measures $π∈ 𝒫(𝓧²)$ with marginals $P$ and $Ꝓ$, that is
	\begin{align}\label{eq:4}
		π(A⨯𝓧)&= P(A) \text{ and } ⏎
		π(𝓧⨯B)&= Ꝓ(B)
	\end{align}
	for all sets $A$ and $B∈ 𝓕$.
	The measures \[	π₁(‧)≔ π(‧⨯𝓧) \text{ and } π₂(‧)≔ π(𝓧⨯‧) \] on $𝓧$ are called the \emph{marginal measures} of the bivariate measure␠$π$.
\end{definition}
We may refer to the excellent monographs \cite{Villani2003, RachevRueschendorf} for a comprehensive discussion of the Wasserstein distance. 
\begin{remark}[Flexibility] \label{rem:170}
	In the subsequent discussion, our problem of interset is to approximate the measure $P$, which is a continuous, a discrete or mixed measure on $𝓧=ℝᵈ$.
	The measure␠$Ꝓ$ is used to approximate the measure␠$P$, which is a discrete measure. The definition of the Wasserstein distance flexibly comprises all the cases, namely continuous, semi-discrete, and discrete measures.
\end{remark}
 In contrast to the standard methodology, we investigate the quantization problem by utilizing an entropy version of the Wasserstein distance. The standard Wasserstein problem is regularized by adding the Kullback–Leibler divergence, which is also known as relative entropy.
\begin{definition}[Kullback–Leibler divergence]
	Let $P$ and $Q∈𝒫 (𝓧)$ be probability measures.
	Denote by $Z∈ L¹(P)$ the Radon–Nikodým derivative, $ⅾQ= Z ⅾP$, if $Q$ is absolutely continuous with respect to $P$ ($Q≪ P$).
	The \emph{Kullback–Leibler divergence} is
	\begin{equation}\label{eq:7-1}
		D(Q∥P)≔ \begin{cases}
			𝔼_P Z㏒Z= 𝔼_Q ㏒Z & \text{if }Q≪ P\text{ and }ⅾQ= Z ⅾP,	⏎
			+∞ & \text{else},
		\end{cases}
	\end{equation}
	where $𝔼_P$ ($𝔼_Q$, resp.\@) is the expectation with respect to the measure $P$ ($Q$, resp.).
\end{definition}
By Gibb’s inequality, the Kullback–Leibler divergence satisfies $D(Q∥P)≥ 0$ (non-negativity). 
However, $D$ is not a distance metric, as it does not satisfy the symmetry, and the triangle inequality properties.

We would like to emphasize the following distinctness to the Wasserstein distance (cf. Remark \ref{rem:170}): for the Kullback–Leibler divergence to be finite ($D(Q∥P)< ∞$), we necessarily have
\[	\supp Q⊂ \supp P, \]
where the support of the measure is (cf. \citet{Rueschendorf})
\[	\supp P≔ ⋂｛A∈ 𝓕∶ A\text{ is closed and }P(A)= 1｝.\]
If $P$ is a continuous measure on $𝓧= ℝᵈ$, then so is $Q$. If $P$ is a finite measure, then the support points of $P$ contain the support points of $Q$.

\subsection{The smooth minimum}
In what follows we present the smooth minimum in its general form, which includes discrete and continuous measures. Numerical computations in the following section rely on results on its discrete version. Therefore, we also address the special properties of its discrete version in detail.

\begin{definition}[Smooth minimum]
	Let $λ> 0$ and~$Y$ be a random variable.
	The \emph{smooth minimum}, or \emph{smooth minimum with respect to $Ꝓ$}, is
	\begin{align}
		\smin{Ꝓ; λ}(Y)	& ≔ -λ㏒𝔼_Ꝓ e^{-Y/λ} ⏎
						& = -λ㏒∫_𝓧 e^{-Y(η)/ λ} Ꝓ(ⅾη),\label{eq:177}
	\end{align}
	provided that the expectation (integral) of $e^{-Y/λ}$ is finite, and $\smin{Ꝓ; λ}(Y)≔ -∞$, if it is not finite.
	For $λ=0$, we set
	\begin{equation}\label{eq:373}
		\smin{Ꝓ; λ=0}(Y)≔ \essinf Y.
	\end{equation}

	For a $σ$‑algebra $𝓖⊂ 𝓕$ and $λ> 0$ measurable with respect to $𝓖$, the \emph{conditional smooth minimum} is 
	\[	\smin{Ꝓ;λ}(Y| 𝓖) ≔ -λ㏒𝔼_Ꝓ❪\left.e^{-Y/λ}\right| 𝓖❫. \]
\end{definition}

The following lemma relates the smooth minimum with the essential infimum (cf. \eqref{eq:373}), that is, colloquially, the ‘minimum’ of a random variable. As well, the result justifies the term \emph{smooth minimum}.
\begin{lemma}
	For $λ> 0$ it holds that 
		\begin{equation}\label{eq:336}
			\smin{Ꝓ;λ}(Y)≤ 𝔼_ꝒY
		\end{equation}
	and 
	\begin{equation}\label{eq:338}
		\essinf Y≤\smin{Ꝓ;λ}(Y)\xrightarrow[λ→0]{} \essinf Y.
	\end{equation}
\end{lemma}
\begin{proof}
	The inequality \eqref{eq:336} follows from Jensen’s inequality, applied to the convex function $x↦ \exp(-x/λ)$. 
	
	Next, the first inequality in the second display \eqref{eq:338} follows from $\essinf Y≤ Y$ and the fact that all operations in \eqref{eq:177} are monotonic.
	Finally, let $a> \essinf Y$. By Markov’s inequality, we have 
	\begin{equation}\label{eq:345}	𝔼_Ꝓ e^{-Y/λ}
		≥ e^{-a/λ} Ꝓ❨e^{-Y/λ} ≥ e^{-a/λ}❩
		= e^{-a/λ} Ꝓ(Y≤ a),
	\end{equation}
	which is a variant of Chernoff’s bound. From inequality \eqref{eq:345}, it follows that 
	\begin{equation}
		\smin{Ꝓ;λ}(Y)
			= -λ ㏒ 𝔼_Ꝓe^{-Y/λ}
			≤ -λ ㏒❨e^{-a/λ} Ꝓ(Y≤ a)❩
			= a+ λ ㏒⅟{Ꝓ(Y≤ a)}.
	\end{equation}
	When $λ>0$ and $λ→ 0$, we have that 
	\[	\smin{Ꝓ;λ}(Y) ≤ a, \]
	where $a$ is an arbitrary number with $a> \essinf Y$. This completes proof.
\end{proof}
\begin{remark}[Nesting property]
	The main properties of the smooth minimum include translation equivariance,
	\[	\smin{Ꝓ;λ}(Y+c)= \smin{Ꝓ;λ}(Y)+c, c∈ ℝ,\]
	and positive homogeneity,
	\[	\smin{Ꝓ;γ‧λ}(γ‧ Y)= γ‧\smin{Ꝓ;λ}(Y), γ> 0.\]
	As a consequence of the tower property of the expectation, we have the nesting property
	\[	\smin{Ꝓ;λ}⟮\smin{Ꝓ;λ}(Y| 𝓖)⟯
		= \smin{Ꝓ;λ}(Y), \]
	provided that $𝓖$ is a sub‑$σ$‑algebra of $𝓕$.
\end{remark}

\subsection{Softmin function}\label{sec:SoftminFunction}
The smooth minimum is related to the softmin function via its derivatives.
In what follows, we express variants of its derivatives, which are involved later.
\begin{definition}[Softmin function]\label{def:Softmin}
	For $λ>0$ and a random variable $Y$ with finite smooth minimum, the \emph{softmin function} is the random variable
	\begin{equation}\label{eq:226}
		σ_λ(Y)≔ \exp❪-{Y- \smin{Ꝓ;λ}(Y) ∕ λ}❫= \frac{e^{-Y/λ}}{𝔼_Ꝓ e^{-Y/λ}},
	\end{equation}
	where the latter equality is obvious with the definition of the smooth minimum in \eqref{eq:177}.
	The function $σ_λ(Y)$ is also called the \emph{Gibbs density}.
\end{definition}
\subsubsection*{The derivative with respect to the probability measure}
The definition of the smooth minimum in␣\eqref{eq:177} does not require the measure $Ꝓ$ to be a probability measure.
Based on ${∂∕∂t}\log(a+t‧h)= {h∕a}$ (at $t=0$) for the natural logarithm, the directional derivative of the smooth minimum in direction of the measure␣$Q$ is
\begin{align}
	⅟t❨\smin{Ꝓ+t‧Q;λ}(Y) - \smin{Ꝓ;λ}(Y)❩
		&= -{λ∕t}❪\log∫_𝓧e^{-Y/λ} ⅾ❨Ꝓ+t‧Q❩- \log∫_𝓧e^{-Y/λ}␠ⅾꝒ❫ ⏎
		&\xrightarrow[t→0]{} -λ‧{∫_𝓧e^{-Y/λ} ⅾQ ∕ ∫_𝓧e^{-Y/λ}␠ⅾꝒ}⏎
		&= -λ‧∫_𝓧σ_λ(Y) ⅾQ\label{eq:412}.
\end{align}
Note, that $-λ σ_λ$ is (up to the constant $-λ$) a Radon–Nikodým density in␣\eqref{eq:412}. The Gibbs density $σ_λ(Y)$ thus is proportional to the directional derivative of the smooth minimum with respect to the underlying measure $Ꝓ$.
\subsubsection*{The derivative with respect to the random variable}
In what follows we shall need the derivative of the smooth minimum with respect to its argument as well. With a similar reasoning as above, this is accomplished by
\begin{align}
	⅟t❨\smin{Ꝓ;λ}(Y+t‧Z) - \smin{Ꝓ;λ}(Y)❩
		&= -{λ∕t}❪\log∫_𝓧e^{-(Y+t‧Z)/λ} ⅾꝒ- \log∫_𝓧e^{-Y/λ}␠ⅾꝒ❫ ⏎
		&= -{λ∕t}❪\log∫_𝓧e^{-Y/λ}⟮1-{t∕λ}Z+ 𝒪(t²)⟯ ⅾꝒ- \log∫_𝓧e^{-Y/λ}␠ⅾꝒ❫ ⏎
		&\xrightarrow[t→0]{} {∫_𝓧Z‧e^{-Y/λ} ⅾꝒ ∕ ∫_𝓧e^{-Y/λ}␠ⅾꝒ}⏎
		&= ∫_𝓧Z‧σ_λ(Y)␠ⅾꝒ,
\end{align}
which involves the softmin function␣$σ_λ(‧)$ as well.

\section{Regularized quantization}\label{sec:Quantization}
This section introduces the entropy regularized optimal quantization problem along with its properties, and recalls the standard optimal quantization problem first.

The standard quantization measures the quality of the approximation by the Wasserstein distance and considers the problem (cf.␠\citet{GrafLuschgy})
\begin{equation}\label{eq:Classical}
	\inf_{π∶\substack{π₁=P, ⏎ π₂∈ 𝒫ₘ(𝓧)}} ∬_{𝓧⨯𝓧}d(ξ,𝜉) π(ⅾξ,ⅾ𝜉),
\end{equation}
where
\begin{equation}\label{eq:468}
	𝒫ₘ(𝓧)≔ ｛Ꝓₘ∈ 𝒫 (𝓧)∶ Ꝓₘ= ∑_{j=1}ᵐ 𝑝ⱼ δ_{yⱼ}｝
\end{equation}
is the set of measures on $𝓧$ supported by not more than $m$ ($m∈ ℕ$) points.

Soft quantization (or quantization, regularized with Kullback–Leibler divergence), instead of␣\eqref{eq:Classical} involves the regularized Wasserstein distance.
The soft quantization problem is regularized with the Kullback–Leibler divergence, it is
\begin{equation}\label{eq:softQuant}
	\inf
		\big\{𝔼_π dʳ+ λ‧D(π∥ P⨯ Ꝓₘ)∶ π₁=P \text{ and }π₂= Ꝓₘ∈ 𝒫ₘ(𝓧)\big\},
\end{equation}
where $λ>0$ and $𝔼_πdʳ= ∬_{𝓧²}d(ξ,𝜉)ʳ π(ⅾξ,ⅾ𝜉)$.
The optimal measure $Ꝓₘ∈ 𝒫ₘ(𝓧)$ solving \eqref{eq:softQuant} depends on the regularization parameter $λ$.


In the following discussion, we initially investigate the regularized approximation, which also demonstrates existence of the optimal approximation.


\subsection{Approximation with inflexible marginal measures}
The following proposition addresses the optimal approximation problem, regularized with Kullback–Leibler divergence and fixed marginals.
To this end, dissect the infimum in the soft quantization problem \eqref{eq:softQuant} as
\begin{equation}\label{eq:303}
	\inf_{Ꝓₘ∈ 𝒫ₘ(𝓧)}\inf_{π∶\substack{π₁=P,⏎ π₂=Ꝓₘ}} 𝔼_π dʳ+λ‧D(π∥ P⨯ Ꝓₘ),
\end{equation}
where the marginals $P$ and $Ꝓₘ$ are fixed in the inner infimum.

The following Proposition \ref{prop:Best-1} addresses this problem with fixed bivariate distribution, which is the inner infimum in␣\eqref{eq:303}.
Then, Proposition␣\ref{prop:444} reveals that the optimal marginals coincide in this case.
\begin{proposition}\label{prop:Best-1}
	Let␣$P$ be a probability measure and $λ> 0$.
	The inner optimization problem in␣\eqref{eq:303} relative to the fixed bivariate distribution $P⨯Ꝓ$ is given by the explicit formula
	\begin{align}\label{eq:Wasserstein-3}
		\inf_{π∶ π₁= P} 𝔼_π dʳ+ λ‧ D(π∥P⨯ Ꝓ)
			&= -λ ∫_𝓧 ㏒∫_𝓧 e^{-d(ξ,𝜉)ʳ/λ} Ꝓ(ⅾ𝜉) P(ⅾξ)	⏎
			&= 𝔼_{ξ∼P}⟮\smin{𝜉∼Ꝓ;λ}d(ξ,𝜉)ʳ⟯.	\label{eq:253}
	\end{align}
	Further, the infimum in␣\eqref{eq:Wasserstein-3} is attained.
\end{proposition}

\begin{remark}
	The notation in␣\eqref{eq:253} (\eqref{eq:260} below, resp.\@) is chosen to reflect the explicit expression␣\eqref{eq:Wasserstein-3}:
	while the soft minimum $\smin{Ꝓ;λ}$ is with respect to the measure $Ꝓ$, which is associated with the variable $𝜉$, the expectation $𝔼_P$ is with respect to $P$, its associated variable is $ξ$
	(that is, the variable␣$ξ$ in␣\eqref{eq:253} is associated with $P$, the variable␣$𝜉$ with␣$Ꝓ$).
\end{remark}
\begin{remark}
	The result␣\eqref{eq:253} extends
	\begin{align}
		\inf_{π∶ π₁=P} 𝔼_π dʳ
				&= ∫_𝓧  \min_{ξ∈ \supp Ꝓ} d(ξ,𝜉)ʳ P(ⅾξ)	\label{eq:259} ⏎
				&= 𝔼_P ⟮\min_{𝜉∈ \supp Ꝓ} d(ξ,𝜉)ʳ⟯,		\label{eq:260}
	\end{align}
	which is the formula without regularization (i.e., $λ=0$, cf.␠\citet{PflugPichlerBuch}).
	Note that the preceding display explicitly involves the support $\supp Ꝓ$, while \eqref{eq:Wasserstein-3} only involves the expectation (via the smooth minimum) with respect to the measure $Ꝓ$.
\end{remark}
\begin{proof}[Proof of Proposition␣\ref{prop:Best-1}]
	It follows from the definition of the Kullback–Leibler divergence in␣\eqref{eq:7-1} that it is enough to consider measures $π$, which are absolutely continuous with respect to the product measure, $π ≪ P\times Ꝓ$; otherwise, the objective is not finite.
	Hence, there is a Radon–Nikodým density $𝑍$ such that, with Fubini's theorem, 
	\[	π(A⨯ B)= ∫_A∫_B 𝑍(ξ,η) Ꝓ(ⅾη)P(ⅾξ). \]
	For the marginal constraint $π(A⨯𝓧)= P(A)$ to be satisfied (cf.␣\eqref{eq:4}), we have that 
	\[	∫_A∫_𝓧 𝑍(ξ,η)Ꝓ(ⅾη) P(ⅾξ)
			= π(A⨯𝓧)= P(A)= ∫_A 1 P(ⅾξ) \]
	for every measurable set $A$. It follows that 
	\[	∫_𝓧 𝑍(ξ,η) Ꝓ(ⅾη) = 1 P(ⅾξ)\text{ almost everywhere}. \]
	We conclude that every density of the form 
	\begin{equation}\label{eq:285}
		𝑍(ξ,η)= \frac{Z(ξ,η)}{∫_𝓧 Z(ξ,η′) Ꝓ(ⅾη′)}
	\end{equation}
	satisfies the constraint␣\eqref{eq:4}, irrespective of␣$Z$ and conversely, every␣$Z$ – via␣$𝑍$ in␣\eqref{eq:285} – defines a bivariate measure $π$ satisfying the constraints␣\eqref{eq:4}.
	We set $Φ(ξ,η)≔ \log Z(ξ,η)$ (with the convention that $\log0=-∞$ and $\exp(-∞)=0$, resp.\@) and consider
	\[	𝑍(ξ,η)= {e^{Φ(ξ,η)} ∕ ∫_𝓧 e^{Φ(ξ,η′)} Ꝓ(ⅾη′)}.\]
	With that, the divergence is 
	\begin{align*}\MoveEqLeft[3]
		D(π∥P⨯ Ꝓ)= ⏎
		& =∫_𝓧∫_𝓧 {e^{Φ(ξ,η)} ∕ ∫_𝓧 e^{Φ(ξ,η′)} Ꝓ(ⅾη′)}
				㏒{e^{Φ(ξ,η)} ∕ ∫_𝓧e^{Φ(ξ,η′)} Ꝓ(ⅾη′)}Ꝓ(ⅾη)P(ⅾξ) ⏎
		& =∫_𝓧∫_𝓧 {e^{Φ(ξ,η)} ∕ ∫_𝓧 e^{Φ(ξ,η′)} Ꝓ(ⅾη′)}Φ(ξ,η) Ꝓ(ⅾη)P(ⅾξ) ⏎
		 & -∫_𝓧\frac{e^{Φ(ξ,η)}}{∫_𝓧 e^{Φ(ξ,η′)} Ꝓ(ⅾη′)}
				㏒∫_𝓧 e^{Φ(ξ,η′)} Ꝓ(ⅾη′)Ꝓ(ⅾη)P(ⅾξ) ⏎
		& =∫_𝓧∫_𝓧\frac{e^{Φ(ξ,η)}}{∫_𝓧e^{Φ(ξ,η′)} Ꝓ(ⅾη′)}Φ(ξ,η)
			-㏒∫_𝓧 e^{Φ(ξ,η′)} Ꝓ(ⅾη′)P(ⅾξ).
	\end{align*}
	For the other term in the objective \eqref{eq:softQuant}, we have 
	\[	𝔼_π dʳ
		= ∫_𝓧∫_𝓧\frac{e^{Φ(ξ,η)}}{∫_𝓧e^{Φ(ξ,η′)} Ꝓ(ⅾη′)} d(ξ,η)ʳ Ꝓ(ⅾη)P(ⅾξ). \]
	Combining the last expressions obtained, the objective in␣\eqref{eq:Wasserstein-3} is
	\begin{align}
		𝔼_π dʳ+ λ D(π∥P⨯ Ꝓ) &
			=∫_𝓧∫_𝓧{e^{Φ(ξ,η)} ∕ ∫_𝓧 e^{Φ(ξ,η′)} Ꝓ(ⅾη′)}⟮d(ξ,η)ʳ
				+ λ Φ(ξ,η)⟯ Ꝓ(ⅾη)P(ⅾξ)\label{eq:9} ⏎
			& -λ∫_𝓧 ㏒∫_𝓧 e^{Φ(ξ,η′)} Ꝓ(ⅾη′)P(ⅾξ)
	\end{align}
	For $ξ$ fixed ($ξ$ is simply suppressed in the following two
	displays to abbreviate the notation), consider the function 
	\[	f(Φ)≔ ∫_𝓧{e^{Φ(η)} ∕ ∫_𝓧 e^{Φ(η′)} Ꝓ(ⅾη′)} ❨d(η)ʳ+ λ Φ(η)❩ Ꝓ(ⅾη)
				-λ ㏒∫_{𝓧}e^{Φ(η′)} Ꝓ(ⅾη′). \]
	The directional derivative in direction $h$ of this function is
	\begin{align}\MoveEqLeft[3]
		\lim_{t→0}⅟t❨f(Φ+t h)-f(Φ)❩= ⏎
		& =∫_𝓧\frac{e^{Φ(η)}}{∫_𝓧e^{Φ(η′)} Ꝓ(ⅾη′)}⟮d(η)ʳ+λ Φ(η)-λ⟯h(η) Ꝓ(ⅾη) ⏎
		& -∫_𝓧\frac{e^{Φ(η)}∫_𝓧 e^{Φ(η′)}h(η′) Ꝓ(ⅾη′)}{❨∫_𝓧 e^{Φ(η′)} Ꝓ(ⅾη′)❩²}⟮d(η)ʳ+λ Φ(η)⟯ Ꝓ(ⅾη)⏎
		& +λ∫_𝓧\frac{e^{Φ(η)}h(η)}{∫_𝓧 e^{Φ(η′)} Ꝓ(ⅾη′)}Ꝓ(ⅾη)	⏎
		& =∫_𝓧\frac{e^{Φ(η)}}{∫_𝓧 e^{Φ(η′)} Ꝓ(ⅾη′)}⟮d(η)ʳ+λ Φ(η)⟯h(η) Ꝓ(ⅾη)\label{eq:14}⏎
		& -∫_𝓧\frac{e^{Φ(η)}∫_𝓧 e^{Φ(η′)}h(η′) Ꝓ(ⅾη′)}{❪∫_𝓧 e^{Φ(η′)} Ꝓ(ⅾη′)❫²}⟮d(η)ʳ+λ Φ(η)⟯ Ꝓ(ⅾη).\label{eq:15}
	\end{align}
	By␣\eqref{eq:14} and␣\eqref{eq:15}, the derivative vanishes for every function $h$, if $d(η)ʳ+ λ Φ(η)= 0$.
	As␣$ξ$ was arbitrary, the general minimum is attained for $Φ(ξ,η)= -d(ξ,η)ʳ/ λ$.
	With that, the first expression in␣\eqref{eq:9} vanishes, and we conclude that 
	\begin{align*}
		\inf_π 𝔼_π dʳ+ λ D(π∥P⨯ Ꝓ) 
			& = -λ∫_𝓧 ㏒∫_𝓧 e^{-d(ξ,η)ʳ/λ} Ꝓ(ⅾη)P(ⅾξ) ⏎
			& = 𝔼_P ❪\smin{Ꝓ;λ}d(ξ,𝜉)ʳ❫.
	\end{align*}

	Finally, notice that the variable $Z(ξ,η)= e^{Φ(ξ,η)}$ is completely arbitrary for the problem␣\eqref{eq:Wasserstein-3} involving the Wasserstein distance and the Kullback–Leibler divergence.
	As outlined above, for every measure $π$ with finite divergence $D(π∥P⨯Ꝓ)$, there is a density␣$Z$ as considered above. 
	With that, the assertion of Proposition␣\ref{prop:Best-1} follows.
\end{proof}

\begin{remark}\label{rem:328}
	The preceding proposition considers probability measures $π$ with marginal $π₁= P$. Its first marginal distribution is (trivially) absolutely continuous with respect to $P$, $π₁≪ P$, as $π₁= P$.

	The second marginal $π₂$, however, is not specified.
	But for $π$ to be feasible in␣\eqref{eq:Wasserstein-3}, its Kullback–Leibler divergence with respect to $P⨯Ꝓ$ is finite.
	There is hence a (non-negative) Radon–Nikodým density␣$Z$ so that 
	\[	π₂(B)= π(𝓧⨯B)= ∬_{𝓧⨯B} Z(ξ,η) P(ⅾξ)Ꝓ(ⅾη).	\]
	It follows from Fubini’s theorem that 
	\[	π₂(B)= ∫_B∫_𝓧 Z(ξ,η) P(ⅾξ)Ꝓ(ⅾη)= ∫_B Z(η) Ꝓ(ⅾη), \]
	where $Z(η)≔ ∫_𝓧 Z(ξ,η) P(ⅾξ)$.
	The second marginal thus is absolutely continuous with respect to $Ꝓ$, $π₂≪ Ꝓ$.
\end{remark}

Proposition␣\ref{prop:Best-1} characterizes the \emph{objective} of the	quantization problem.
Its proof, implicitly, reveals the marginal of the best approximation as well. The following lemma spells out the density of the marginal of the optimal measure with respect to $Ꝓ$ explicitly.
\begin{lemma}[Characterization of the best approximating measure]\label{lem:356}
	The best approximating marginal probability measure minimizing \eqref{eq:Wasserstein-3} has density
	\[	Z(𝜉)= 𝔼_P σ_λ❪d(ξ,𝜉)ʳ❫= ∫_𝓧 σ_λ❪d(ξ,𝜉)ʳ❫ P(ⅾξ), \]
	where $σ_λ(‧)$ is the softmin function (cf.␠Definition \ref{def:Softmin}).
\end{lemma}
\begin{proof}
	Recall from the proof or Proposition␣\ref{prop:Best-1} the density
	\[	𝑍(ξ,𝜉)
			= {e^{-d(ξ,𝜉)ʳ/λ} ∕ 𝔼_Ꝓ e^{-d(ξ,𝜉)ʳ/λ}} \]
	of the optimal measure $π$ relative to $P⨯Ꝓ$. From that we derive that 
	\[	π₂(B)= π(𝓧⨯ B)
			= ∫_B∫_𝓧 {e^{-d(ξ,𝜉)ʳ/λ} ∕ 𝔼_Ꝓe^{-d(ξ,𝜉)ʳ/λ}} P(ⅾξ)Ꝓ(ⅾ𝜉) \]
	so that 
	\[	Z(𝜉)= ∫_𝓧{e^{-d(ξ,𝜉)ʳ/λ} ∕ 𝔼_Ꝓe^{-d(ξ,𝜉)ʳ/ λ}}P(ⅾξ)
			= 𝔼_P σ_λ❨d(ξ,𝜉)ʳ❩ \]
	is the density with respect to $Ꝓ$, that is $ⅾπ₂=Z ⅾꝒ$ (i.e., $π₂(ⅾ𝜉)=Z(𝜉) Ꝓ(ⅾ𝜉)$).
\end{proof}

\subsection{Approximation with flexible marginal measure}
The following proposition reveals that the best approximation of a bivariate measure in terms of a product of independent measures is given by the product of its marginals.
With that it follows that the objectives in␣\eqref{eq:303} and␣\eqref{eq:Wasserstein-3} coincide for $Ꝓ=π₂$.
\begin{proposition}\label{prop:444}
	Let $P$ be measure and $π$ be a bivariate measure with marginal $π₁=P$ and $π₂$. Then it holds that
	\begin{equation}\label{eq:446}
		D(π∥ P⨯π₂) ≤ D(π∥ P⨯Ꝓ),
	\end{equation}
	where $Ꝓ$ is an arbitrary measure.
\end{proposition}
\begin{proof}
	Define the Radon–Nikodým density $Z(η)≔ {π₂(ⅾη) ∕ Ꝓ(ⅾη)}$ and observe that the extension $Z(ξ,η)≔ Z(η)$ to $𝓧⨯𝓧$ is the density $Z= {ⅾP⨯π₂ ∕ ⅾP⨯Ꝓ}$.
	It follows with␣\eqref{eq:7-1} that 
	\begin{align}\label{eq:452}
		0 	≤ D(π₂∥ Ꝓ)&= 𝔼_{π₂}\log {ⅾπ₂∕ⅾꝒ} ⏎
		  & = 𝔼_π\log{ⅾ P⨯π₂ ∕ ⅾ P⨯ Ꝓ} ⏎
		  & = 𝔼_π❨\log{ⅾ π ∕ ⅾ P⨯Ꝓ}-\log{ⅾ π ∕ ⅾ P⨯π₂}❩ ⏎
		  &	= D(π∥ P⨯Ꝓ) - D(π∥ P⨯π₂),
	\end{align}
	which is the assertion.
	In case the measures are not absolutely continuous, the assertion␣\eqref{eq:452} is trivial.
\end{proof}
Suppose now that $π$ is a solution of the master problem \eqref{eq:Wasserstein-3} with some $Ꝓ$.
It follows from the preceding proposition that the objective \eqref{eq:Wasserstein-3} improves when replacing the initial $Ꝓ$ by the marginal of the optimal solution, $Ꝓ= π₂$.

\subsection{The relation of soft quantization and entropy}
	The soft quantization problem \eqref{eq:Wasserstein-3} involves the Kullback–Leibler divergence and \emph{not} the entropy.
	The major advantage of the formulation presented above is that it works for discrete, continuous or mixed measures, while entropy usually needs to be defined separately for discrete and continuous measures.

	For a discrete measure with $P(x)≔ P(❴x❵)$ and $Ꝓ(y)≔ Ꝓ(❴y❵)$, the Kullback–Leibler divergence \eqref{eq:7-1} is
	\begin{align}
		D(Ꝓ∥ P)&= H(Ꝓ,P)- H(Ꝓ)	\label{eq:631}	⏎
			&= ∑_{x∈𝓧}Ꝓ(x)\log{Ꝓ(x) ∕ P(x)},
	\end{align}
	where 
		\[	H(Ꝓ,P)≔ -∑_{x∈𝓧} Ꝓ(x)‧㏒ P(x) \]
	is the \emph{cross-entropy} of the measures␣$Ꝓ$ and␣$P$, and \begin{equation}\label{eq:676}
		H(Ꝓ)≔ H(Ꝓ,Ꝓ) = -∑_{x∈𝓧} Ꝓ(x)㏒Ꝓ(x)
	\end{equation} the \emph{entropy} of␣$Ꝓ$. 

	For a measure $π$ with marginals $P$ and $Ꝓ$, the cross-entropy is 
	\begin{align}
		H(π, P⨯Ꝓ)&= -∑_{x,y}π(x,y)㏒❨P(x)‧Ꝓ(y)❩⏎
			& =		-∑_{x,y}π(x,y)㏒ P(x)-∑_{x,y}π(x,y)㏒ Ꝓ(y) ⏎
			& =		-∑ₓP(x)㏒ P(x)-∑_y Ꝓ(y)㏒ Ꝓ(y), \label{eq:283}
	\end{align}
	where we have used the marginals \eqref{eq:4}.
	Note, that \eqref{eq:283} does not depend on $π$, and hence $H(π, P⨯Ꝓ)$ does not depend on $π$.

	With␣\eqref{eq:631}, the quantization problem \eqref{eq:Wasserstein-3} thus rewrites equivalently as
	\begin{equation} \label{eq:Wasserstein-2}
		\min_{π∶ π₂∈ P}∬_{𝓧⨯𝓧}dʳ ⅾπ- λ‧H(π)
	\end{equation}
	by involving the entropy only. 
	For this reason, we shall call the master problem~\eqref{eq:Wasserstein-3} also the \emph{entropy regularized problem}.

\section{Soft tessellation}\label{sec:Tessellation}
The quantization problem \eqref{eq:303} consists in finding a good (in the best case the optimal) approximation of a general probability measure $P$ on $𝓧$ by a simple, and discrete measure $Ꝓₘ= ∑_{j=1}ᵐ 𝑝ⱼ δ_{yⱼ}$.
The problem thus consists in finding good weights $𝑝₁,…,𝑝ₘ$, as well as good locations $y₁,…,yₘ$.
Quantization employs the Wasserstein distance to measure the quality of the approximation;
soft quantization involves the regularized Wasserstein distance, instead (as in \eqref{eq:Wasserstein-3}):
\[	\inf_{Ꝓₘ∈ ℘ ₘ(𝓧)} \inf_{π∶\substack{π₁=P,⏎ π₂=Ꝓₘ}} 𝔼_π dʳ+λ‧D(π∥ P⨯ Ꝓₘ),\]
where the measures on $𝓧$ supported by not more than $m$ points are (cf. \eqref{eq:468})
\[	𝒫ₘ(𝓧)= ｛Ꝓₘ∈ ℘ (𝓧)∶ Ꝓₘ= ∑_{j=1}ᵐ 𝑝ⱼ δ_{yⱼ}｝. \]

We separate the problem of finding the best weights and locations.
The following Section \ref{sec:Weights} addresses the problem of finding the optimal weights $𝑝$, the subsequent Section \ref{sec:Location} then the problem of finding the optimal locations $y₁,…,yₘ$.
As well, we shall elaborate the numerical advantages of \emph{soft} quantization below.

\subsection{Optimal weights}\label{sec:Weights}
Proposition \ref{prop:Best-1} above is formulated for general probability measures $P$ and $Ꝓ$. The desired measure in quantization is a simple and discrete measure.
To this end recall that measures, which are feasible for \eqref{eq:Wasserstein-3}, have marginals $π₂$ with $π₂≪ Ꝓ$ by Remark \ref{rem:328}. It follows that the support of the marginal is smaller than the support of $Ꝓ$, that is
\[	\supp π₂ ⊂ \supp Ꝓ. \]
For a simple measure $Ꝓ= ∑_{j=1}ᵐ 𝑝ⱼ δ_{yⱼ}$ with $𝑝ⱼ>0$, it follows in particular that $\supp π₂⊂ ❴y₁,…,yₘ❵$. We consider the measure $Ꝓ$ and the support $❴y₁,…,yₘ❵$ fixed in this subsection.

To unfold the result of Proposition \ref{prop:Best-1} for discrete measures we recall the smooth minimum and the softmin function for the discrete (empirical or uniform) measure $Ꝓ= ∑_{j=1}ᵐ𝑝ⱼ δ_{yⱼ}$.
For this measure, the smooth minimum \eqref{eq:177} explicitly is
\[	\smin{λ;Ꝓ}(y₁,…,yₘ)
		= -λ\log⟮𝑝₁ e^{-y₁/λ}+…+𝑝ₘ e^{-yₘ/λ}⟯. \]
For $λ= 1$ and uniform weights $𝑝₁=… = 𝑝ₘ= ⅟m$, this quantity is occasionally referred to as
\emph{\href{https://en.wikipedia.org/wiki/LogSumExp}{LogSumExp}}.
The softmin function (or Gibbs density \eqref{eq:226}) is
\[	σ_λ(y₁,…,yₘ) = ❪{e^{-yⱼ/λ} ∕ 𝑝₁ e^{-y₁/λ}+… +𝑝ₘ e^{-yₘ/λ)}}❫_{j=1}ᵐ. \]

It follows from Lemma \ref{lem:356} that the best approximating measure is
$Q= ∑_{j=1}ᵐqⱼ 𝑝ⱼ δ_{yⱼ}$, where the vector $q$ of optimal weights, relative to $Ꝓ$, is given explicitly by
\begin{equation}\label{eq:442}
	q= ∫_𝓧 σ_λ❨d(ξ,y₁)ʳ,…, d(ξ,yₘ)ʳ❩ P(ⅾξ)= 𝔼_P σ_λ❨d(ξ,y₁)ʳ,…, d(ξ,yₘ)ʳ❩,
\end{equation}
which involves computing expectations.

\subsubsection*{Soft tessellation}
	For $λ=0$, the softmin function $σ_λ$ is
	\[	𝑝ⱼ‧σ_{λ=0}❨d(ξ,y₁)ʳ,…,d(ξ,yₘ)ʳ❩ⱼ=
		\begin{cases}
			1 & \text{if } d(ξ,yⱼ)ʳ= \min ❨d(ξ,y₁)ʳ,…,d(ξ,yₘ)ʳ❩, ⏎
			0 & \text{else.}
	\end{cases}\]
	That is, the mapping $j↦ 𝑝ⱼ‧σ_λ(…)ⱼ$ can serve for classification, i.e., tessellation: 
	the point $ξ$ is associated to $yⱼ$, if $σ_λ(…)ⱼ≠ 0$ and the corresponding region is known as Voronoi diagram.

	For $λ>0$, the softmin $𝑝ⱼ‧σ_λ(…)ⱼ$ is not a strict indicator, but can be interpreted as probability instead. That is,
	\[	𝑝ⱼ‧σ_λ⟮d(ξ,y₁)ʳ,…, d(ξ,yₘ)ʳ⟯ⱼ \]
	is the probability of allocating $ξ∈ 𝓧$ to the quantizer $yⱼ$.

\subsection{Optimal locations\label{sec:Location}}
As a result of Proposition \ref{prop:Best-1}, the objective in \eqref{eq:253} is an expectation.
To identify the optimal support points $y₁,…,yₘ$, it is central to minimize
\begin{equation}\label{eq:512}
	\min_{Ꝓ=∑_{j=1}ᵐ 𝑝ⱼδ_{yⱼ}} 𝔼_{ξ∼P}⟮\smin{λ;y∼Ꝓ}d(ξ,y)ʳ⟯.
\end{equation}
This is a stochastic, non-linear and non-convex optimization problem.
\begin{equation}\label{eq:Wasserstein}
	f(y₁,…,yₘ)≔ 𝔼 f(y₁,…,yₘ;ξ)= 𝔼⟮\smin{λ;Ꝓ}_{j=1,…,m}d(ξ,yᵢ)ʳ⟯,
\end{equation}
where the function $f(y₁,…,yₘ;ξ)≔ \smin{λ;Ꝓ} ❴d(ξ,yᵢ)ʳ∶ j=1,…,m❵$ is non-linear and non-convex;
the optimal quantization problem constitutes an unconstrained, stochastic, non-convex and non-linear optimization problem.
The gradient of the objective is built of the components
\begin{equation}\label{eq:767}
	{∂∕∂yⱼ}f(y₁,…,yₘ)
		= {𝑝ⱼ‧\exp❨-d(ξ,yⱼ)ʳ/ λ❩ ∕ ∑_{j′= 1}ᵐ𝑝_{j′}‧\exp❨-d(ξ,y_{j′})ʳ/ λ❩}
			‧ \left. ∇_y d(ξ,y)ʳ\right|_{y=yⱼ},
\end{equation}
that is,
\begin{equation}\label{eq:785}
	∇f = 𝑝• σ_λ❨d(ξ,y₁)ʳ,…,d(ξ,yₘ)ʳ❩
			• r d(ξ,y)^{r-1}
			• ∇_y d(ξ,y),
\end{equation}
where ‘$•$’ denotes the Hadamard (element-wise) product and $𝑝$, $d(ξ,y)^{r-1}$ are the vectors with entries $𝑝ⱼ$, $d(ξ,yⱼ)^{r-1}$, $j=1,…,m$.

Algorithm \ref{alg:quantizers} is a stochastic gradient algorithm to minimize \eqref{eq:442}, which collects the elements of the optimal weights and the optimal locations given in the preceding and this section.

\begin{algorithm}[t]
	\KwResult{Optimal quantizing measure $Ꝓ= ∑_{j=1}ᵐ𝑝ⱼ δ_{yⱼ}$ with optimal weights $𝑝$ and locations $y$}
	\KwData{A sequence $(αₖ)_{k=1}^∞$ (the learning rate) with $∑_{k=1}^∞ αₖ= ∞$ and $∑_{k=1}^∞ αₖ²< ∞$, for example $αₖ= {σ ∕ (30+k)^{\nicefrac23}}$, where $σ²$ is the variance of the measure $P$;␤
	$𝑝$ probability weights guess with $𝑝ⱼ>0$ ($j=1,…,m$) and $𝑝₁+…+𝑝ₘ=1$, for example $𝑝← ❨⅟m,…⅟m❩$;}

	set $k← 0$ \hfill initialize the iteration count

	{\Repeat{desired approximation quality achieved}
		{	$y_{k+1}← yₖ- αₖ‧∇f(yₖ)$ \hfill update the location by stochastic approximation (cf. \eqref{eq:785}) ⏎
			$𝑝_{k+1}← {k ∕ k+1}𝑝ₖ + ⅟{k+1}𝑝ₖ•σ_λ(yₖ)$ \hfill update the probabilities ⏎
		$k← k+1$
	}}
	\KwRet{$y=(y₁,…yₘ)$, $𝑝= (𝑝₁,…,𝑝ₘ)$}
	\caption{Stochastic gradient algorithm to find the optimal quantizers and the optimal masses\label{alg:quantizers}}
\end{algorithm}
\begin{example}
	To provide an example for the gradient of the distance function in \eqref{eq:767} (\eqref{eq:785}, resp.), the derivative of the weighted norm 
	\[	d(ξ,y)=∥y-ξ∥ₚ≔ ❪∑_{ℓ=1}ᵈ w_ℓ‧|y_ℓ-ξ_ℓ|ᵖ❫^{\nicefrac1p}\]
	is 
	\[	\frac{∂}{∂yⱼ}∥y-ξ∥ₚʳ
		= r wⱼ ∥ξ-y∥ₚ^{\frac{r-p}p}‧|yⱼ-ξⱼ|^{p-1}‧\sign(yⱼ-ξⱼ).\]
\end{example}


\subsection{Quantization with large regularization parameters}
The entropy in \eqref{eq:676} is minimal for the Dirac measure $P=δ_x$ (where $x$ is any point in $𝓧$): in this case, $H(δ_x)= 1‧㏒1= 0$, while $H(Ꝓ)>0$ for any other measure.
For larger values of $λ$, the objective in \eqref{eq:Wasserstein-2} – and thus the objective of the master problem \eqref{eq:468} – supposedly will give preference to measure with fewer points.
This is indeed the case, as Theorem \ref{thm:198} (above) states.
We give its proof below, after formally defining the center of the measure.
\begin{definition}[Center of the measure]\label{def:center}
	Let $P$ be a probability measure on $𝓧$ and $d$ be a distance on $𝓧$.
	The point $a∈ 𝓧$ is a \emph{center of the measure~$P$ with respect to the distance~$d$}, if 
	\[	a∈ \argmin_{x∈ 𝓧} 𝔼 d(x,ξ)ʳ,\]
	provided that $𝔼 d(x₀,ξ)ʳ< ∞$ for some (and thus any) $x₀∈ 𝓧$ and $r≥1$.
\end{definition}
In what follows, we demonstrate that the regularized quantization problem~\eqref{eq:Wasserstein-2} links the optimal quantization problem and the center of the measure.
\begin{proof}[Proof of Theorem \ref{thm:198}]
	The problems \eqref{eq:Wasserstein-2} and \eqref{eq:Wasserstein-3} are equivalent by Proposition \ref{prop:Best-1}.
	Now assume that $yᵢ=yⱼ$ for all $i$, $j≤ m$, then $d(yᵢ,ξ)= d(yⱼ,ξ)$ for $ξ∈ Ξ$, and it follows that
	\[	\smin{λ}❨d(y₁,ξ)ʳ,…,d(yₘ,ξ)ʳ❩= d(yᵢ,ξ)ʳ,  i=1,…,m.\]
	The minimum of the optimization problem thus is attained at $yᵢ=a$, for each $i=1,…,m$, where~$a$ is the center of the measure $P$ with respect to the distance~$d$.
	It follows that $y₁=… =yₘ =a$ is a local minimum and a stationary point, satisfying the first order conditions
	\[	∇f(y₁,…,yₘ)= 0\]
	for the function $f$ given in␣\eqref{eq:Wasserstein}.
	Note as well that
	\[	σ_λ❨d(ξ,y₁)ʳ,…,d(ξ,yₙ)ʳ❩ᵢ
			= \frac{\exp⟮-d(ξ,yᵢ)ʳ/λ⟯}{∑_{j=1}ⁿ𝑝ⱼ\exp⟮-d(ξ,yⱼ)ʳ/λ⟯}
			= 1, \]
	the softmin function does not depend on $λ$ at the stationary point $y₁=… = yₘ= a$.
	
	Recall from␣\eqref{eq:767} that
	\[	∇𝔼⟮\smin{λ;Ꝓ}_{j=1,…,n}d(ξ,yⱼ)ʳ⟯
			= 𝔼 σ_λ❨d(y₁,ξ)ʳ,…,d(yₙ,ξ)ʳ❩ \boldsymbol{‧}∇ d(ξ,yᵢ)ʳ.\]
	By the product rule, the Hessian matrix is
	\begin{equation}
		∇²𝔼❪\smin{λ;Ꝓ}_{j=1,…,n}d(ξ,yⱼ)ʳ❫=
			𝔼\begin{pmatrix}\begin{array}{l}
				∇σ_λ❨d(y₁,ξ)ʳ,…,d(yₙ,ξ)ʳ❩•❨∇d(ξ,yᵢ)ʳ❩²⏎
				 + σ_λ❨d(y₁,ξ)ʳ,…,d(yₙ,ξ)ʳ❩•∇²d(ξ,yᵢ)ʳ
		\end{array}\end{pmatrix}.\label{eq:3}
	\end{equation}
	Note that the second expression is positive definite, as the Hessian $∇²d(ξ,yᵢ)ʳ$ of the convex function is positive definite and $∇\smin{λ;Ꝓ}_{j=1,…,n}(x₁,…,xₙ)= σ_λ(x₁,…,xₙ)≥ 0$.
	Further, the Hessian of the smooth minimum (see also the appendix) is 
	\[	∇σ_λ= ∇²\smin{λ}_{j=1,…,n}= -⅟{λ} Σ, \]
	where the matrix $Σ$ is
	\[	Σ≔ \diag\bigl(σ₁,…,σₙ\bigr)- σσਾ. \]
	This matrix␣$Σ$ is positive definite (as $∑_{i=1}ⁿσᵢ= 1$) and $0≤ Σ≤ 1$ in Loewner order (indeed, $Σ$ is the covariance matrix of the multinomial distribution).
	It follows that the first term in~\eqref{eq:3} is $𝓞(1)$, while the second is $𝓞❨⅟{λ}❩$, so that␣\eqref{eq:3} is positive definite for $λ$ sufficiently small.
	That is, the extremal point $yᵢ= a$ is a minimum for all $λ$.
	In particular, there exists $λ₀> 0$ such that \eqref{eq:3} is positive definite for every $λ> λ₀$ and hence the result.
\end{proof}
 
\section{Numerical illustration}\label{sec:Numerical}
This section presents numerical findings for the approaches and methods discussed earlier.
The Julia implementations for these methods are available online.\footnote{Cf.\ \href{https://github.com/rajmadan96/SoftQuantization.git}{https://github.com/rajmadan96/SoftQuantization.git}}

\medskip
In the following experiments, we approximate the measure $P$ by a finite discrete measure $\tilde{P}$ using the stochastic gradient algorithm, Algorithm~\ref{alg:quantizers}.
\subsection*{One dimension}
First, we perform the analysis in one dimension. In this experiment, our problem of interest is to find entropy regularized optimal quantizers for
\[P \sim \mathcal{N}(0,1) \quad \text{and}\quad P \sim \operatorname{Exp}(1)\] (the normal and the exponential distribution with standard parameters).
To enhance the peculiarity, we consider only $m=8$ quantizers.

Figure␣\ref{fig:1ab} illustrates the results of soft quantization of standard normal distribution and exponential distribution. 
It is apparent that when $λ$ is increased beyond a certain threshold (cf. Theorem \ref{thm:198}), the quantizers converge towards the center of the measure (i.e., the mean), while for smaller values of $λ$, the quantizers are able to identify the actual optimal locations with greater accuracy. 
Furthermore, we want to emphasize that our proposed method is capable of identifying mean location regardless of the shape of the distribution, which this experiment empirically substantiates.
\begin{figure}[ht!]
	\centering
	\subfloat[][Normal distribution: for $λ=10$, the best approximation residues at the center of the measure; for $\lambda=1$, the approximation reduces to 6 points only (two remaining points have probability $0$); for $λ=0$, we obtain the standard quantization]
	{	\includegraphics[width=0.48\textwidth]{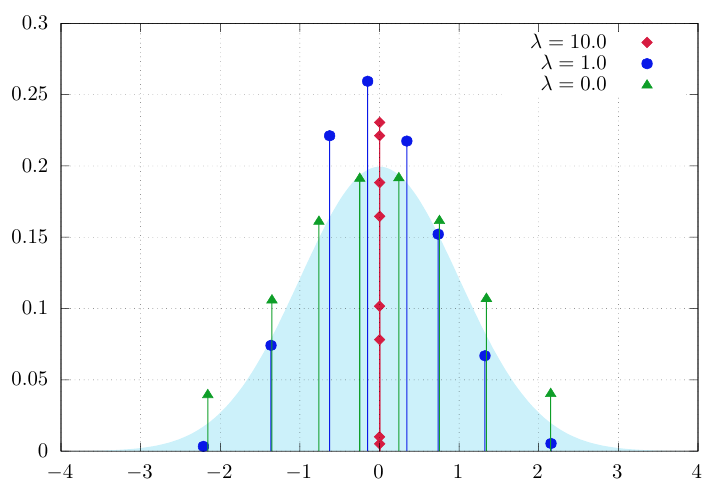}\label{fig:2a}}\hfill
	\subfloat[][Exponential distribution: the measures concentrate on 1 ($λ$␣large), 3 ($λ=1$), 5 ($λ=0.5$) and 8 quantization points]
	{	\includegraphics[width=0.48\textwidth]{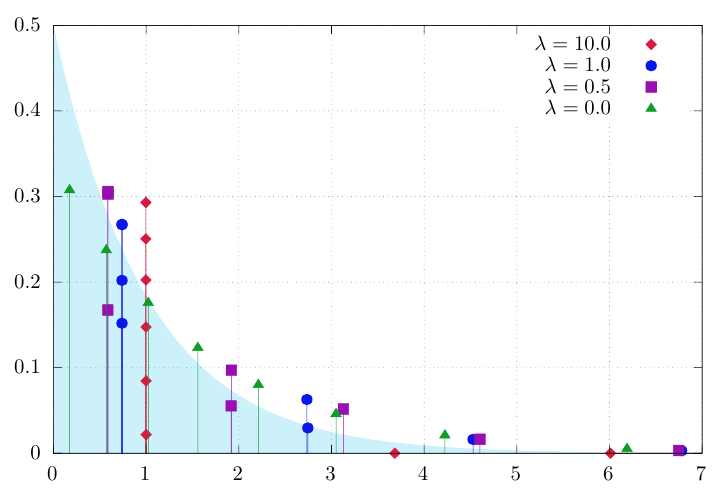}\label{fig:2b}}
	\caption{Soft quantization of measures on $ℝ$ with varying regularization parameter $λ$ with 8 quantization points\label{fig:1ab}}
\end{figure}
\begin{figure}[ht!]
	\centering
	\subfloat[][$λ=0$: approximate solution to standard quantization problem with 8 quantizers]
	{	\includegraphics[width=0.30\textwidth,height=0.23\textwidth]{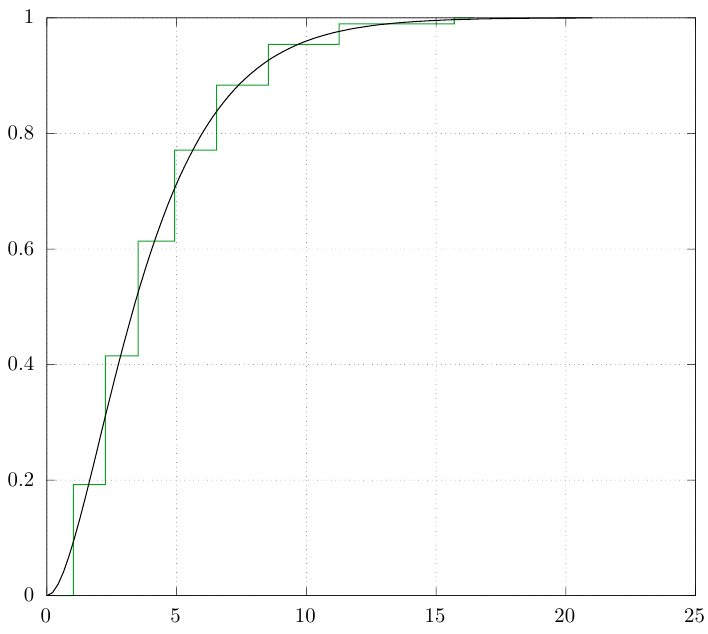}\label{fig:6a}}

	\subfloat[][$λ=1$: the 8 quantization points collapse to 7 quantization points]
	{	\includegraphics[width=0.30\textwidth,height=0.23\textwidth]{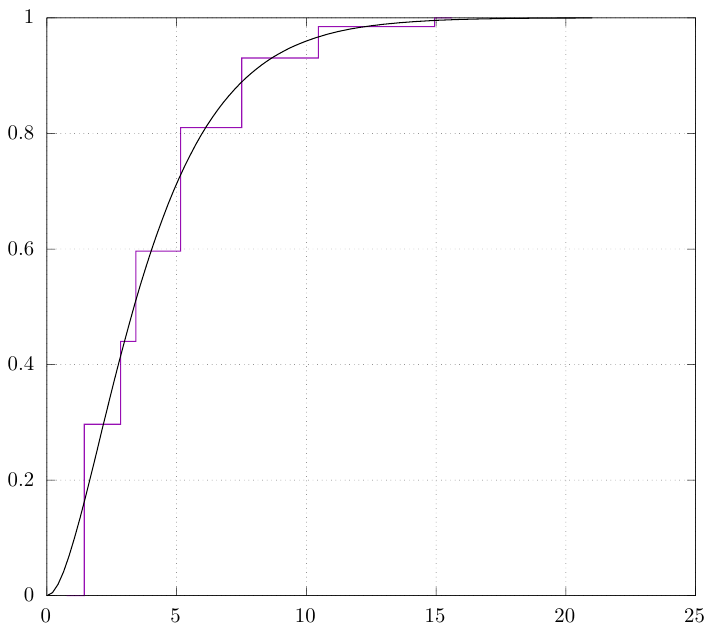}\label{fig:6b}}
	\hfill
	\subfloat[][$λ=10$: 8 quantization points collapse to 3 quantization points]
	{	\includegraphics[width=0.30\textwidth,height=0.23\textwidth]{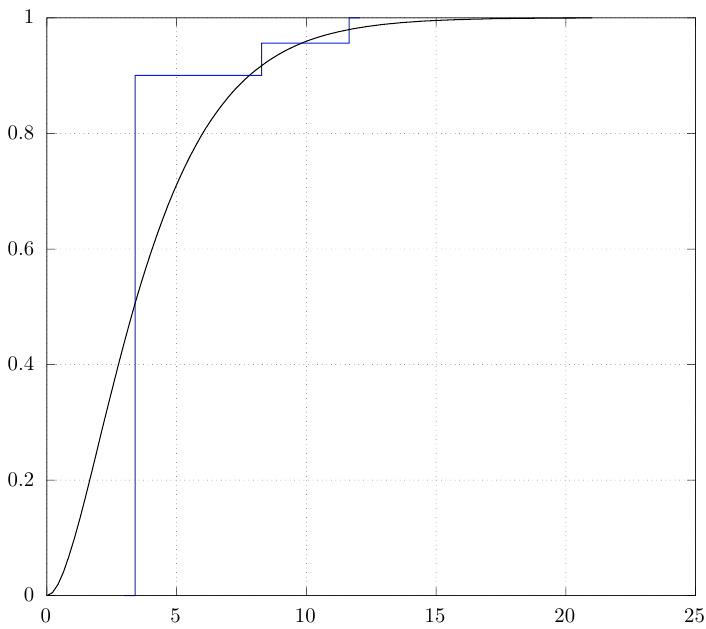}\label{fig:6c}}
	\hfill
	\subfloat[][$λ=20$: the quantization points converge to one single point, the center of the measure]
	{	\includegraphics[width=0.30\textwidth,height=0.23\textwidth]{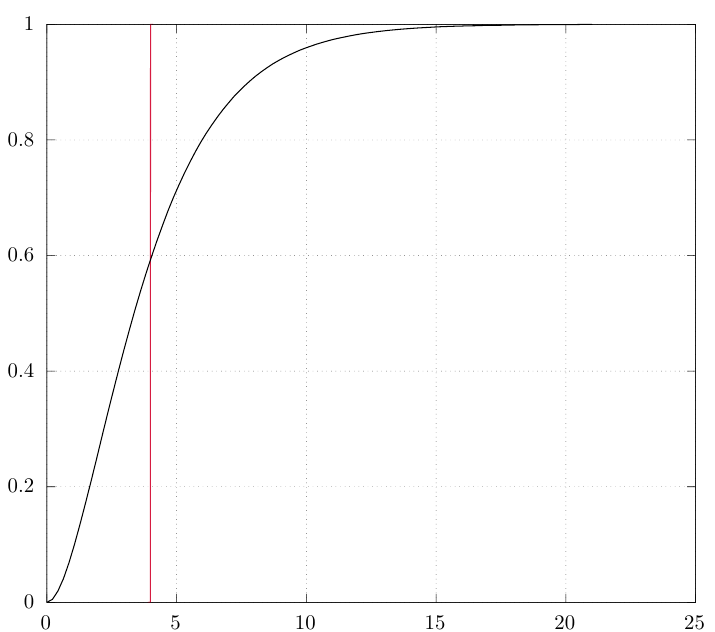}\label{fig:6d}}
	\caption{Soft quantization of the Gamma distribution on $ℝ$ with varying regularization parameter $λ$; the approximating measure simplifies with $λ$ increasing}\label{fig:1DUniform_8}
\end{figure}

In order to increase the understanding of dissemination of weights (probabilities) and their respective positions, the following examination involves the calculation of the cumulative distribution function.
Additionally, we consider  \[P \sim \Gamma(2, 2) \text{(Gamma distribution)}\] as a problem interest, a notably distinct scenario in terms of shape compared to the measures previously examined.

Figure~\ref{fig:1DUniform_8} provides results.
It is evident that as $λ$ increases, the number of quantizers $m$ decreases.
When $λ$ reaches a specific threshold, such as with $λ=20$ in our case, all quantizers converge towards the center of the measures, represented by the mean (i.e., 4).
\subsection*{Two  dimensions}
Next, we demonstrate the behavior of entropy regularized optimal quantization for a range of $λ$ in two dimensions. In the following experiment, we consider  
\[P \sim U❨(0,1)\times(0,1)❩ \text{(uniform distribution on the square)}\]
as a problem of interset. Initially, we perform the experiment with $m=4$ quantizers.
\begin{figure}[ht!]
	\subfloat[][Uniform distribution in $ℝ²$]
	{	\includegraphics[width=0.49\textwidth]{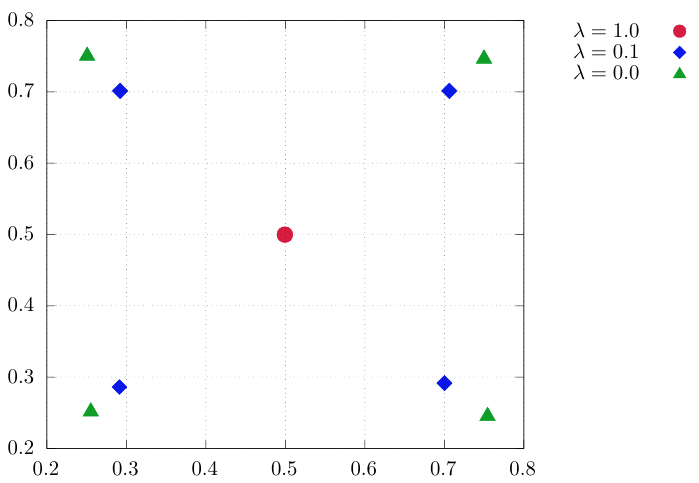}\label{fig:3a}}
	\subfloat[][Enlargement of Figure~\ref{fig:3a}: for larger values of $λ$ (here, $λ=1$), the quantizers align while converging to the center of the measure]
	{	\includegraphics[width=0.5\textwidth]{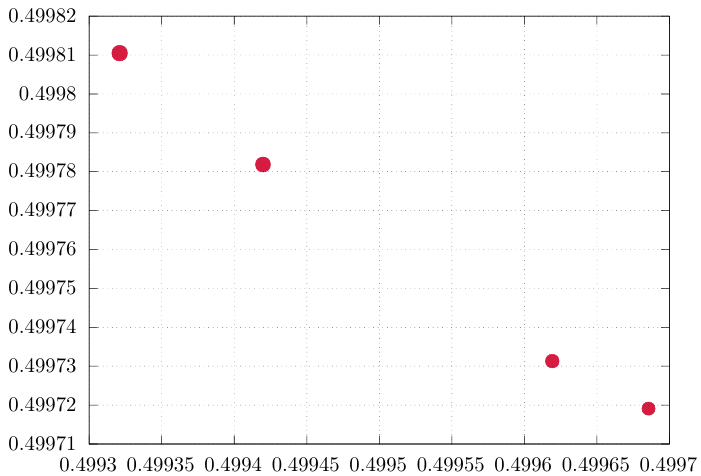}\label{fig:3b}}
	\caption{Two dimension--Soft quantization of uniform distribution on $ℝ²$ with varying regularization parameter $λ$ with $4$ quantizers}\label{fig:2DUniform}

\end{figure}
Figure~\ref{fig:2DUniform} illustrates the findings. Subplot~\ref{fig:3a} reveals a quantization pattern similar to what was observed in the one-dimensional experiment. However, in Subplot~\ref{fig:3b}, we gain a detailed insight into the behavior of quantizers at $λ=1$, where they align diagonally before eventually colliding. Furthermore, the size of the point indicates the respective probability of the quantization point, which is notably uniformly distributed for varying regularization parameter $λ$. 
\begin{figure}[ht!]
	\centering
	\subfloat[][$λ=0.0$: approximate solution to standard quantization problem with 16␣quantizers]
	{	\includegraphics[width=0.30\textwidth,height=0.23\textwidth]{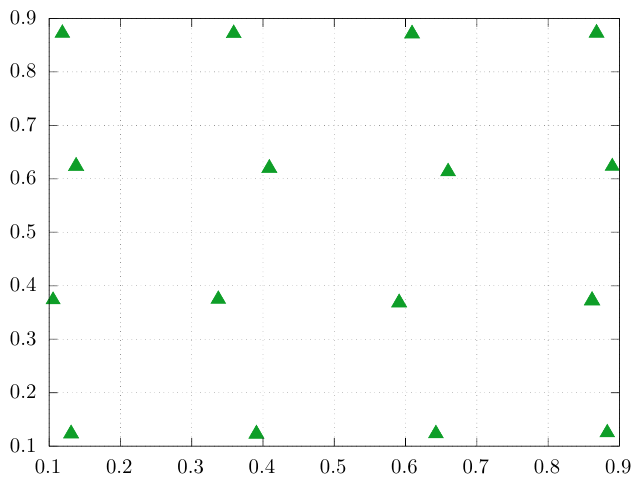}\label{fig:4a}}

	\subfloat[][$λ=0.037$: the 16 quantization points collapse to 8 quantization points]
	{	\includegraphics[width=0.30\textwidth,height=0.23\textwidth]{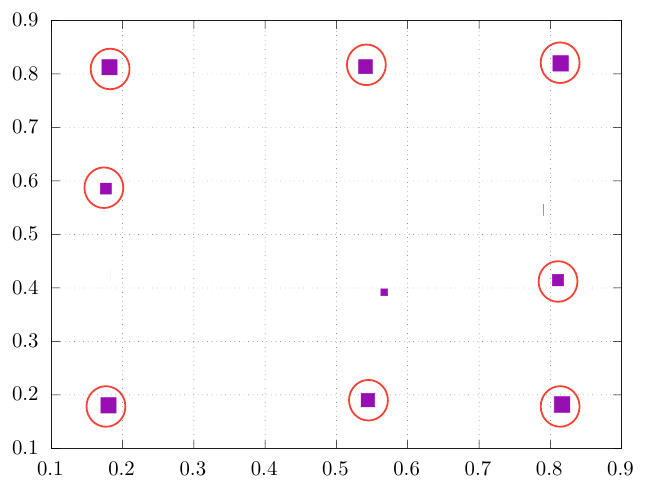}\label{fig:4b}}
	\hfill
	\subfloat[][$λ=0.1$: the 16 quantization points collapse to 4 quantization points]
	{	\includegraphics[width=0.30\textwidth,height=0.23\textwidth]{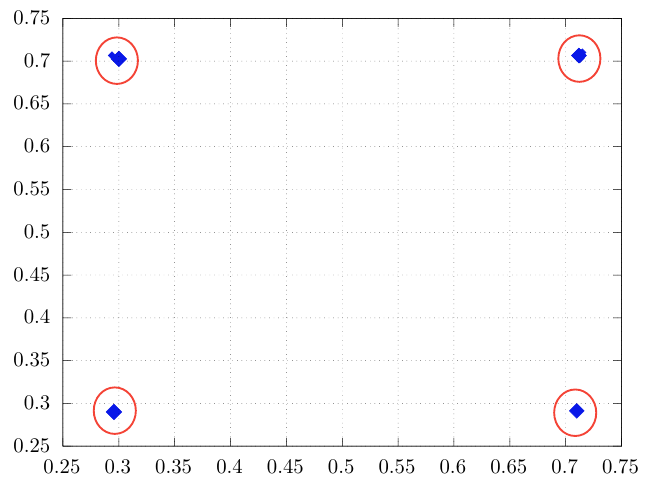}\label{fig:4c}}
	\hfill
	\subfloat[][$λ=1.0$: the quantization points converge to one single point, the center of the measure, in an aligned way]
	{	\includegraphics[width=0.30\textwidth,height=0.23\textwidth]{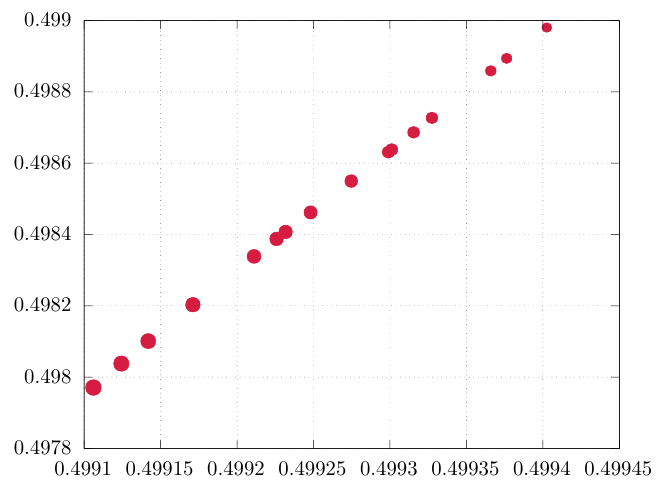}\label{fig:4d}}
	\caption{Soft quantization of uniform distribution on $ℝ²$ with varying regularization parameter $λ$; the approximating measure simplifies with $λ$ increasing}\label{fig:2DUniform_16}
\end{figure}

Once again, we are considering uniform distribution as a problem of interest in the subsequent experiment, this time employing $m=16$ quantizers for enhanced comprehension. 
Figure~\ref{fig:2DUniform_16} encapsulates the essence of the experiment, offering an extensive visual representation.  
In contrast to the previous experiment, we observe that for regularization values of $λ=0.037$  and $λ=0.1$, they assemble at the nearest strong points (in terms of high probability) rather than converging toward the center of the measure (see Subplots~\ref{fig:4b} and \ref{fig:4c}).
Subsequently, for larger $λ$, they move from these strong points toward the center, where they make a diagonal alignment before collision (see Subplot~\ref{fig:4d}).
More concisely, when $λ=0$, we achieve the genuine quantization solution (see Subplot~\ref{fig:4a}). As $λ$ increases, quantizers with lower probabilities converge towards those with nearest higher probabilities. Subsequently, all quantizers converge towards the center of the measure, represented by the mean of respective measure.
\begin{figure}[ht!]
	\centering
	\subfloat[][$λ=0.0$ (solution to standard quantization problem)]
	{	\includegraphics[width=0.30\textwidth,height=0.25\textwidth]{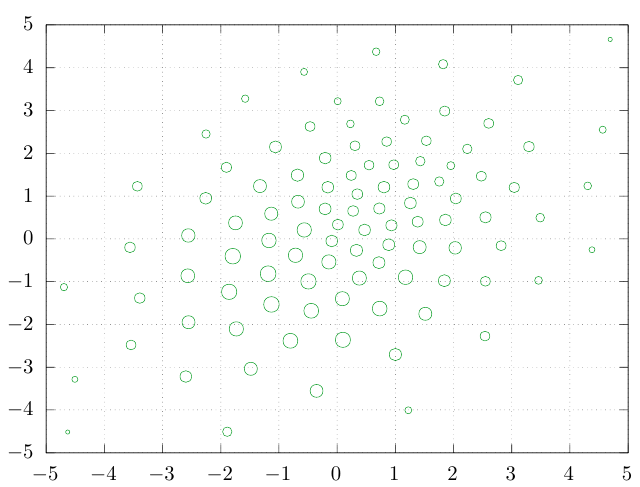}\label{fig:5a}}
	\hfill
	\subfloat[][$λ=5.0$]
	{	\includegraphics[width=0.30\textwidth,height=0.25\textwidth]{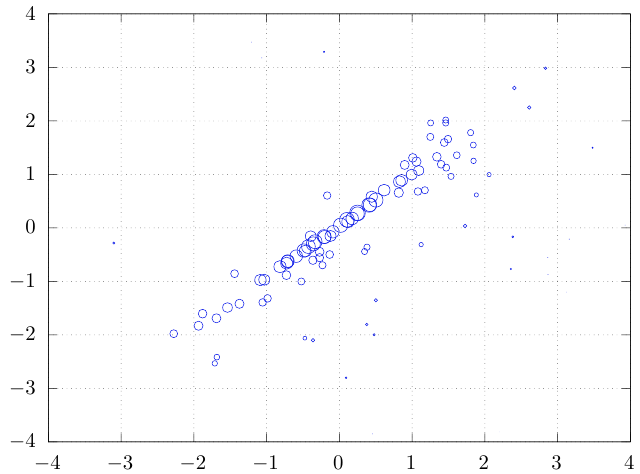}\label{fig:5b}}
	\hfill
	\subfloat[][$λ=10.0$]
	{	\includegraphics[width=0.30\textwidth,height=0.25\textwidth]{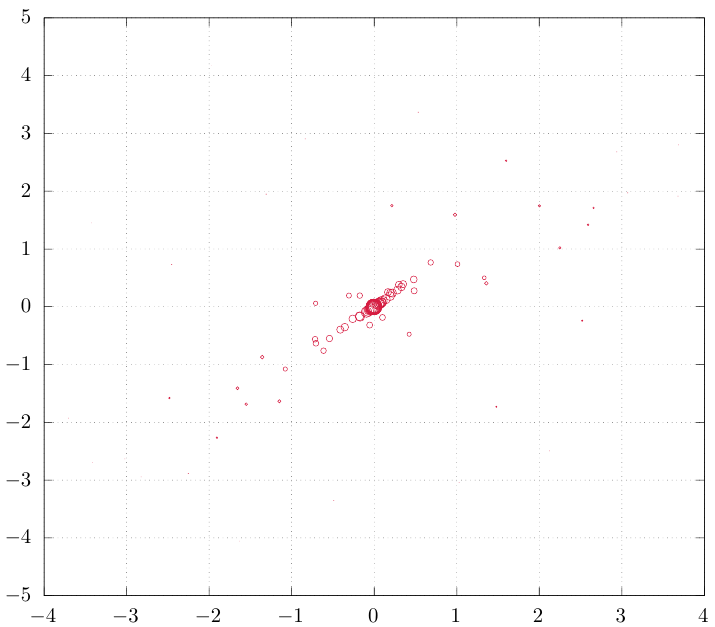}\label{fig:5c}}
	\caption{Two dimension--Soft quantization of normal distribution on $ℝ²$ with varying regularization parameter $λ$; parameters: $r=2$, $p=2$ and $m=100$}\label{fig:2DNormal_100}
\end{figure}

Thus far, we have conducted two-dimensional experiments, employing various quantizers ($m=4$ and $m=16$) with the uniform distribution. Now, we will delve into the complexity of a multivariate normal distribution, aiming to enhance comprehension.
More precisely, our problem of interest is to find soft quantization for 
\[P \sim \mathcal N(\mathbf{\mu}, \Sigma),\]
where \[\mu = \begin{pmatrix} 0 \\ 0 \end{pmatrix}, \quad
\Sigma = \begin{pmatrix} 3 & 1 \\
						 1  & 3 \end{pmatrix}.\]
In this endeavor, we employ more quantizers, specifically $m=100$.
Figure~\ref{fig:2DNormal_100} captures the core essence of the experiment, delivering a comprehensive and visually illustrative representation.
From the experiment it becomes evident that, as $λ$ increases, the initial diagonal alignment precedes convergence toward the center of the measure.
Additionally, we observe a noticeable shift of points with lower probabilities towards those with higher probabilities.
Furthermore, this experiment highlights that the threshold of $λ$ for achieving convergence or diagonal alignment in the center of the measure is dependent on the number of quantizers employed.

\section{Summary}\label{sec:Summary}
This study enhances the stability and simplicity of the standard quantization problem by introducing a novel method of quantization using entropy. 
Propositions~\ref{prop:Best-1} and \ref{prop:444} thoroughly elucidate the intricacies of the master problem~\eqref{eq:303}.
Our substantiation of convergence of quantizers to the center of measure explains the transition from a complex hard optimization problem to a simplified configuration (see Theorem~\ref{thm:198}). 
More concisely, this transition underscores the fundamental shift towards a more tractable and straightforward computational framework, marking a significant advancement in our approach.
Moreover, in Section~\ref{sec:Numerical}, we provided numerical illustrations of our method, thereby confirming the robustness, stability, and properties as discussed in our theoretical results. These numerical demonstrations serve as empirical evidence reinforcing the efficacy of our proposed approach.

\bibliographystyle{abbrvnat}
\bibliography{LiteraturAlois,LiteraturRaj}

\appendix

\section*{Appendix\label{sec:Appendix}}
\section{Hessian of the softmin}

%
	
	The empirical measure $\frac1{λ}\sum_{i=1}^{n}\delta_{x_{i}}$ is a probability measure.
	Form Jensen’s inequality, it follows that $\smin{\lambda}(x_{1},\dots,x_{n})\le\frac{1}{n}\sum_{i=1}^{n}x_{i}\eqqcolon\overline{x}_{n}$.
	The smooth minimum thus involves a cumulant generating function for which we derive that 
	\begin{align}
		\smin{λ}(x₁,…,xₙ) & =∑_{j=1}^∞ {(-1)^{j-1} ∕ λ^{j-1}‧ j!}κⱼ ⏎
	 & =\overline{x}ₙ- {1∕ 2λ}sₙ²+⅟{6λ²}κ₃+𝓞(λ^{-3}),\label{eq:7}
	\end{align}
	where $κⱼ$ is the $j$‑th cumulant with respect
	to the empirical measure. Specifically,
	\[	κ₁= \overline{x}ₙ= ⅟n∑_{i=1}ⁿxᵢ, 
		κ₂= sₙ²= ⅟n∑_{i=1}ⁿ(xᵢ-\overline{x}ₙ)², 
		κ₃= ⅟n∑_{i=1}ⁿ(xᵢ- \overline{x}ₙ)³, \]
	where $\overline{x}ₙ$ is the ‘sample mean’ and $sₙ²$ the ‘sample variance’, the following cumulants ($κ₄$, etc.)\ are more involved, though.
	The Taylor series expansion $\log(1+x)= x-⅟2 x²+𝓞(x^3)$ and
	\begin{align*}
		-λ\log⅟n∑_{i=1}ⁿe^{-xᵢ/λ}
			& =\overline{x}ₙ-λ\log⅟n∑_{i=1}ⁿe^{-(xᵢ-\overline{x}ₙ)/λ}⏎
			& =\overline{x}ₙ-λ\log∑_{i=1}ⁿ ⅟n❨1-\frac{xᵢ-\overline{x}ₙ}{λ}
				 +⅟2❨\frac{xᵢ-xₙ}{λ}❩²
				 -⅟6❨\frac{xᵢ-\overline{x}ₙ}{λ³}❩
				+𝓞❨⅟{λ⁴}❩❩⏎
			& =\overline{x}ₙ- λ\log❨1+ ⅟{2λ²}sₙ²- ⅟{6λ³}κ₃+ 𝓞(λ^{-4})❩⏎
			& =\overline{x}ₙ- λ❨⅟{2λ²}sₙ²- ⅟{6λ³}κ₃+ 𝓞(λ^{-4})❩⏎
			& =\overline{x}ₙ- ⅟{2λ}sₙ²+ ⅟{6λ²}κ₃+𝓞(λ³).
	\end{align*}
	Note as well that the softmin function is the gradient of the smooth
	minimum,
	\[	σ_λ(x₁,…,xₘ)ᵢ= \frac{∂}{∂xᵢ}\smin{λ}(x₁,…,xₙ). \]

The softmin function is frequently used in classification in a maximum likelihood framework. 
It holds that
\begin{align*}
	{∂²∕∂xᵢ∂xⱼ}\smin{λ}(x₁,…,xₙ)
	& ={∂∕∂xⱼ}\frac{\exp(-λ xᵢ)}{∑_{k=1}ᵐ\exp(-λ xₖ)}⏎
	& =+ λ\frac{\exp(-λ xᵢ-λ xⱼ)}{\left(∑_{k=1}ᵐ\exp(-λ xₖ)\right)²}⏎
	& = λ σᵢσⱼ
\end{align*}
for $i≠j$ and 
\begin{align*}
	{∂²∕∂xᵢ²}\smin{λ}(x₁,‥,xₙ)
	& ={∂∕∂xᵢ}\frac{\exp(-λ xᵢ)}{∑_{j=1}ᵐ\exp(-λ xⱼ)}⏎
	& =\frac{-λ\exp(-λ xᵢ)\left(∑_{j=1}ᵐ\exp(-λ xⱼ)\right)+λ\exp(-λ xᵢ-λ xᵢ)}{\left(∑_{j=1}ᵐ\exp(-λ xⱼ)\right)}\\
	& =-λ σᵢ+λσᵢσᵢ,
\end{align*}
that is,
\[	∇²\smin{λ}(x₁,…,xₙ)
	= λ⟮σσਾ-\diag σ⟯
	=-λ‧\left(\left(\begin{array}{ccc}
			σ₁ & 0 & ⋱ \\
			0  & ⋱ & 0\\
			⋱  & 0 & σₙ
		\end{array}\right)- σ‧σਾ \right). \]
\end{document}